\documentclass[a4paper,11pt]{article}
\usepackage[latin1]{inputenc}
\usepackage[english]{babel}
\usepackage{hyperref}
\usepackage{amsmath,amsfonts,amscd,amssymb}
\usepackage{amsthm}

\usepackage[all]{xy}
\usepackage{graphicx,epsfig}
\usepackage{makeidx}
\usepackage{color}
\newtheorem{thm}{Theorem}[section]
\newtheorem{deft}[thm]{Definition}

\newtheorem{lem}[thm]{Lemma} 
\newtheorem{prop}[thm]{Proposition}

\newtheorem{remark}[thm]{Remark}

\def\Z{\mathbb{Z}} 
\def\R{\mathbb{R}} 
\def\C{\mathbb{C}} 

\newcommand{\Cc}{\textbf{C}}

\renewcommand {\epsilon}{\varepsilon}
\renewcommand {\le}{\leqslant}
\renewcommand {\ge}{\geqslant}

\renewcommand {\geq}{\geqslant}

\renewcommand{\thefigure}{\arabic{figure}}

\makeatletter
\renewcommand{\fnum@figure}{{\figurename~\thefigure\ }}
\makeatother

\def\figurename{{Fig.}}%

\makeindex
\title{Jet schemes and minimal toric embedded resolutions of rational double point singularities }\author{Hussein Mourtada, Camille Pl\'enat}

\begin{document}

\maketitle

\begin{abstract}Using the structure of the jet schemes of rational double point singularities, we construct ''minimal embedded toric resolutions'' of these singularities. 
We also establish, for these singularities, a correspondence  between a natural class of irreducible components of the jet schemes centered at the singular locus and 
the set of divisors which appear on every ''minimal embedded toric resolution''. 
 We prove that this correspondence is bijective except for the $E_8$ singulartiy. This can be thought as an embedded Nash correspondence for rational double
 point singularities. 
\end{abstract}


\thispagestyle{empty}

\thispagestyle{empty}






\section{Introduction}

\let\thefootnote\relax\footnote{\noindent \textbf{2010 Mathematics Subject Classification.} 14E15,14E18,14M25.\\
\textbf{Keywords} Embedded Nash problem, Resolution of 
singularities, Toric Geometry.\\
This research was partially supported by the ANR-12-JS01-0002-01 SUSI. The second author is supported by Labex Archim\`ede  (ANR-11-LABX-0033) and  the A*MIDEX project (ANR-11-IDEX-0001-02), 
funded by the ``Investissements d'Avenir" French Government programme managed by the French National Research Agency (ANR)", and by the CLVA.

}

In this article, we construct embedded resolutions  of surfaces having rational double point singularities; also called simple singularities. 
The word simple refers to the fact that they have no moduli ({i.e.,} a hypersurface singularity with the same topological type of a simple singularity
is analytically isomorphic to the simple singularity \cite{LeT}).
But at the same time, they are simple from a resolution of singularities point of view, i.e., easy to resolve. The traditional approach
to resolve singularities  is to iterate blowing ups at smooth centers in order to make an invariant drop. This invariant 
takes values in a discrete ordered set with a smallest element (which detects smoothness). It should not only detect smoothness,  but also be easy to  
compute so that its behavior can be followed when iterating the blowing ups. In this article, we adopt a different strategy to resolve a simple singularity $X\subset \Cc^3:$
we construct an embedded toric resolution of $X\subset \Cc^3.$ This construction is based on a deep invariant, the set \textbf{EC} (essential components) of irreducible components of the jet schemes satisfying some natural properties. 
Recall that for $m\geq 0,$ the $m-$th jet scheme of a variety $X$ parametrizes morphisms $\mbox{Spec}~\Cc[t]/(t^{m+1})\longrightarrow X$ (see section 2 for details).  
These are finite dimensional approximations of the space of arcs which parametrize germs of curves drawn on $X.$ Actually the geometry of the $m-$th jet scheme is 
intimately related with the geometry of the set of arcs in a smooth ambient space containing $X$ and which have ''contact'' with $X$ larger than $m.$ This explains 
in parts why while the arc space of $X$ detect information about abstract resolution of singularities {(see \cite{N})}, jet schemes detect information about embedded resolution of singularities {(see \cite{Mo5}, \cite{LMR}, \cite{ELM})}.


The other subject which this article considers is the minimality of embedded resolutions of singularities. 
In contrast with the abstract resolution case, there is no ''universal'' minimal embedded resolution for surface singularities. Therefore we need  to
make this notion more precise. Actually, the embedded resolution of 
$X\subset \Cc^3$ that we construct is toric and is obtained from a particular regular subdivision of the Newton dual fan $\Gamma$ associated with $X$ (see section 3 for the definition of $\Gamma$). Since simple singularities are Newton non-degenerate (see section 3), regular subdivisions of $\Gamma$ 
give toric embedded resolutions of $X\subset \Cc^3$ (see \cite{Ho},\cite{Va},\cite{LJ}). This is equivalent to say that an abstract resolution of singularities of the toric variety $Z_\Gamma$
defined by the fan $\Gamma$ gives an embedded toric resolution of $X\subset \Cc^3.$ But the toric variety $Z_\Gamma$ is of dimension $3$ and
hence thanks to a theorem by Bouvier and Gonzalez-Sprinberg \cite{BGS}, we know that there exists a toric resolution of singularities of $Z_\Gamma$ where all the irreducible 
divisors of the exceptional locus are essential, i.e., the centers of the divisorial valuations associated with these divisors give irreducible components 
of the exceptional locus of any other toric resolution of singularities of $Z_\Gamma.$ We call such a resolution of singularities a minimal toric resolution; in \cite{BGS}, it 
is called $G-$desingularization; we call a toric embedded resolution of $X\subset \Cc^3$ minimal if it corresponds to a minimal toric resolution of $Z_\Gamma.$
Note that in general, a minimal toric embedded resolution of singularities is not unique, but the divisorial valuations associated with the irreducible divisors of 
its exceptional locus are the same for all minimal toric embedded resolution of singularities. We say that these last divisorial valuations are \textit{embedded essential}.\\
We will prove that the toric embedded resolution of $X\subset \Cc^3$ which we construct from jet schemes is minimal for all simple singularities except for $E_8$ (Note that the $E_8$ singularity behaves exceptionally
also from many other points of view  \cite{LeT2}). In particular we describe a bijection between the set of essential components of the jet schemes of simple singularities (except for $E_8$) and the set of embedded essential 
divisorial valuations. This can be thought as a solution of an embedded {version of the} Nash problem. \\

The choice of this class of singularities is related to the following facts:\\

First, for $m$ big enough the number of irreducible components of their $m-$th jet scheme is constant; this simplifies the classification of the irreducible components 
we are interested in, when $m$ varies. This is not the case in general {(see \cite{Mo2}, \cite{Mo3})}.

Second, these singularities are Newton non-degenerate with respect to their Newton polyhedron, and therefore they have a toric embedded resolution; see \cite{AGS},\cite{O1},\cite{O2}. Defining the 
class of irreducible components of jet schemes, mentioned above, is more subtle for Newton degenerate singularities \cite{LMR}. This is related to a conjecture of Teissier on embedding any singularity in 
such way that it can be resolved by a toric morphism {(see \cite{T1},  \cite{T2}, \cite{Mo5})}.\\

We would like to thank Shihoko Ishii, Monique Lejeune-Jalabert, Patrick Popescu-Pampu and Bernard Teissier for several discussions about
this article.

\section{Jet schemes}

In this section, we begin by giving some preliminaries on jet schemes; we then recall from \cite{Mo1}  the structure of jet schemes of rational double point singularities 
and we extract from this stucture some information about particular irreducible components of these jet schemes. These information will be used in the next section 
to obtain the canonical resolution of these singularities. We will all the cases of simple singularities (see their defining equations below); the case of $E_6$ 
(see below the defining equation) has been treated in \cite{Mo1}, 
but we consider it {briefly here}  for the convenience of the reader. \\

Let $k$ be an algebraically closed field of arbitrary characteristic and $X$ be a $k$-algebraic variety. For $m \in \mathbb{N},$ the functor $$F_m :k-Schemes \longrightarrow Sets$$

 $$Spec(A)\longrightarrow Hom_k(Spec A[t]/(t^{m+1}),X)$$
where $A$ is a $k-$algebra, is representable by a $k-$scheme $X_m$ \cite{Is}. $X_m$ is the m-th jet scheme of $X$, and $F_m$ is  isomorphic to its functor of points.
In particular the closed points of $X_m$ are in bijection with the $k[t]/(t^{m+1})$ points of $X$.
\\For $m,p \in \mathbb{N}, m > p$, the truncation homomorphism $A[t]/(t^{m+1}) \longrightarrow A[t]/(t^{p+1})$ induces
a canonical projection $\pi_{m,p}: X_m \longrightarrow X_p.$ These morphisms verify $\pi_{m,p}\circ \pi_{q,m}=\pi_{q,p}$
for $p<m<q$, and they are affine morphisms, they define a projective system whose limit is a scheme that we denote by $X_{\infty}.$ This is the arc space of $X$.
\\ Note that $X_0=X$. We denote the canonical projection $\pi_{m,0}:X_m\longrightarrow X_0$ by $\pi_{m},$ and denote by $\Psi_m$  the canonical morphisms $X_{\infty}\longrightarrow X_m.$ \\

We now assume that $X\subset \Cc^3$ is defined by one of the following equations:
$$ A_n,~n \in \mathbb{N}: xy-z^{n+1}=0.$$  
$$ D_n,~n \in \mathbb{N},n\geq 4: z^2-x(y^2+x^{n-2})=0.$$  
$$ E_6 : z^2+y^3+x^4=0^.$$
$$ E_7 : x^2+y^3+yz^3=0.$$ 
$$ E_8 : z^2+y^3+x^5=0^.$$
We denote by $X_m^0:=\pi_{m,0}^{-1}(O),$ where $O$ is the origin of $\Cc^3;$ it is the singular locus of $X.$

We will now associate a divisorial valuation over $\Cc^3$ with  
the irreducible component of $\mathcal{C}_m\subset X_m^0$ satisfying the 
property $(\star)$ that for every irreducible component 
$\mathcal{C}_{m+1}\subset X^0_{m+1},$
$$\text{if}~~\pi_{m+1,m}(\mathcal{C}_{m+1})\subset \mathcal{C}_m~~ 
\text{then}~~ 
\text{codim}(\mathcal{C}_{m+1})>\text{codim}(\mathcal{C}_{m}).~~~~(\star)$$
For $m\in \mathbb{N},$ let $\psi_m^{a}: \Cc^3_\infty \longrightarrow \Cc^3_m$ be the canonical morphism, here the exponent $a$ stands for {"ambient"}.
For $p\in \mathbb{N},$ we consider the following cylinder in the arc space
$$Cont^p(f)=\{\gamma \in \Cc^3_\infty ;~~ \mbox{ord}_tf\circ \gamma=p\}.$$
Since $\psi_m^{a}$ is a trivial fibration, for every  irreducible component $\mathcal{C}_{m}\subset X_m^0,$ we have that
$$ {\psi_{m}^{a}}^{-1}(\mathcal{C}_{m})\cap Cont^{m+1}(f)$$ is an irreducible component of $Cont^{m+1}(f),$ whenever this last intersection is non-empty. But
from the property $(\star)$ we have  $\text{codim}(\mathcal{C}_{m+1})>\text{codim}(\mathcal{C}_{m}),$ this implies  $$ {\psi_{m}^{a}}^{-1}(\mathcal{C}_{m})\cap~ Cont^{m+1}(f)\not=\emptyset.$$\\
Let $\gamma$ be the generic point of   $ {\psi_m^{a}}^{-1}(\mathcal{C}_{m})\cap ~Cont^{m+1}(f),$
then for every $h\in \Cc[x,y,z],$ we define the map 
$\nu_{\mathcal{C}_{m}}:\Cc[x,y,z]\longrightarrow \mathbb{N},$ as follows
\begin{equation}\label{val}
 \nu_{\mathcal{C}_{m}}(h)=\mbox{ord}_th\circ \gamma.
\end{equation}

It follows from corollary $2.6$ in \cite{ELM}, that $\mathcal{C}_{m}$ is a divisorial valuation (see also \cite{dFEI}, \cite{Re}, prop. 3.7 (vii) applied to ${\psi_{m}^{a}}^{-1}(\mathcal{C}_{m})$).\\ 

Given $m \geq 1,$ with an irreducible component $\mathcal{C}_m$ of $X_m^0,$ we associate the following vector that we call the weight vector:
$$v(\mathcal{C}_m)=(\nu_{\mathcal{C}_m}(x),\nu_{\mathcal{C}_m}(y),\nu_{\mathcal{C}_m}(z))   \in \mathbb{N}^{3}.$$ 
In the following definition, we consider the irreducible components of $X_m^0$ that will be meaningful for the problem we are considering.
\begin{deft}
The following set of particular irreducible components  of  $X_m^0:$
$$EC(X):=\{\mathcal{C}_m\subset X_m^0,m\geq 1 ~~ \text{is an irreducible 
component satisfying $(\star)$ and}$$
\begin{equation} \label{ec}v(\mathcal{C}_m)\not=v(\mathcal{C}_{m-1})~~
\text{for {any component}} ~~\mathcal{C}_{m-1} ~~\text{verifying}~~ 
\pi_{m,m-1}(\mathcal{C}_m)\subset \mathcal{C}_{m-1} \},
\end{equation}
is called the set of essential components of $X.$ \\
We also consider the following set of associated valuations :
$$EE(X):=\{\nu_{\mathcal{C}_m},~~\mathcal{C}_m\in EC(X) \},$$
Where $EE$ stands for Embedded-Essential, to say that these valuations are supposed to appear {in} every embedded toric resolution.

\end{deft}
Note that the definitions of the sets $EC(X)$ and $EE(X)$ are ad-hoc to simple singularities.
Indeed{,} the fact that these singularities are non-degenerate 
with respect to their Newton polyhedron
implies that the meaningful valuations   are monomial{;} this 
motivates the definition above of the set $EC(X).$ From another point of view, if we check 
carefully the equations, we will figure out that the components which does not 
belong to the set $EC(X),$ are at their generic points trivial fibrations above
components which belong to $EC(X);$ that is why they are not useful 
and this phenomenon reflects the fact that the singularities we consider are
Newton non-degenerate.\\

Note also that on one hand  we are interested in the embeddings of the
singularities in $\Cc^3$ that are defined by the equations we gave at
the beginning of this section; this includes that we consider the variables 
$x,y,z$ as given and hence the embedding of the torus in $\Cc^3$ is 
also given. On the other hand, since we are interested in toric resolutions, the 
exceptional locus of such a resolution will be defined as the inverse image of 
the complement of the torus, i.e., of the coordinate hyperplanes. So the fact that  
the $x$ axis (defined by $y=z=0$ in $\Cc^3$) and the $y$ axis (defined by $x=z=0$ in $\Cc^3$) belong to the singularities $A_n,$ the $y$ axis 
belongs to the singularities $D_n$ and the $z$ axis belongs to the singularity  
$E_7,$ suggests that if we want to find a toric resolution, we need also to 
consider the jets which go through respectively these coordinate axis. But the family of $m-$th jets 
whose center is a generic point on one of these axis is irreducible because 
our surfaces have isolated singularities at the origin. So the components 
which come from those families are easy to determine; again since our   
singularities are Newton non-degenerate, we are only interested in the vectors
associated with these components, so we only consider such components for small 
$m,$ whenever the associated vector changes. For $m$ large enough, the 
associated vector stabilizes. \\

A careful reading of \cite{Mo4} (and of  \cite{Mo2})  produces a determination of the set $EC$ for 
a rational double point $(X,0).$ We will denote it by $EC(X).$  Let 
$f(x,y,z) \in \Cc[x,y,z]$ be the defining equation of a rational double 
point singularity $X.$ We write
$$f(\sum_{i=0}^{m}x_it^i,\sum_{i=0}^{m}y_it^i,\sum_{i=0}^{m}z_it^i)=\sum_{i=0}^{i=m}F_it^i~~~ mod~~t^{m+1},~~~~~~~(\diamond)$$
then the $m$-th jet scheme $X_m$ of $X$ is naturally embedded in $$\Cc^{3(m+1)}=(\Cc^{3})_m=\text{Spec}\Cc[x_i,y_i,z_i,i=0,\ldots,m]$$
and is defined by the ideal
$I_m=(F_0,F_1,...,F_m).$ From Section $3$ in \cite{Mo4} and section $3.2$ in 
\cite{Mo1}, we are able to determine the set $EC(X)$. This is the subject of 
the  following lemma. We do not treat the $E_8$ singularity (which will be in 
some sense an exception to part of our formulation) because it takes too much 
space; we only give the weight vectors associated with the components in 
$EC(E_8))$ in the corollary that follows the lemma.

\begin{lem}\label{ec}
For an $A_n$  singularity, $EC(A_n)$ is given by the 
components centered at the 
origin $$\{V(x_0,\ldots,x_{l-1},y_0,\ldots,y_{m-l},z_0), 
l=1,\ldots,m~~\mbox{and}~~m=1,\ldots, n\},$$
and the components centered at the $x$ respectively $y$ axis are
$$\{V(y_0,\ldots,y_{l},z_0), 
l=0,\ldots,n\},$$
respectively $$\{V(x_0,\ldots,x_{l},z_0), 
l=0,\ldots,n\}.$$

For a  $D_{2n}$ singularity, $EC(D_{2n})$ is given by the components centered at the 
origin $$\{X_1^0=V(x_0,y_0,z_1),$$
$$X_2^0=V(x_0,y_0,z_0,z_1),$$
$$H_{2k}=V(x_0,y_0,z_0,z_1,y_1,y_2,...,y_{k-1},z_2,z_3,...,z_k),k=2,\ldots,n-1,$$

$$H_{2k+1}=V(x_0,y_0,z_0,z_1,y_1,y_2,...,y_{k},z_2,z_3,...,z_k),k=1,\ldots,n-2,$$

$$L_{2k+1}=V(x_0,x_1,y_0,z_0,z_1,y_1,y_2,...,y_{k-1},z_2,z_3,...,z_{k}),k=1,
\ldots,2n-1\},$$

and the components centered at the $y$-axis are 
$$V(x_0,z_0),~~\mbox{and}~~V(x_0,x_1,z_0).$$
For an $E_6$ singularity we have  
$$EC(E_6)=\{V(x_0,y_0,z_0),V(x_0,y_0,z_0,z_1),V(x_0,y_0,
y_1,z_0,z_1),$$
$$V(x_0,x_1,y_0,y_1,z_0,z_1,z_2),V(x_0,x_1,
y_0 ,y_1,y_2,z_0,z_1,z_2,z_3),$$
$$V(x_0,x_1,x_2,y_0
,y_1,y_2,y_3,z_0,z_1,z_2,z_3,z_4,z_5)\}.$$

For an $E_7$ singularity, 
$EC(E_7)$ is given by the components centered at the 
origin $$\{V(x_0,y_0,z_0),V(x_0,x_1,y_0,z_0),V(x_0,x_1,y_0,y_1,z_0),
V(x_0,x_1,x_2,y_0,y_1,z_0),$$
$$V(x_0,x_1,x_2,y_0,y_1,y_2,z_0),V(x_0,x_1,x_2,y_0,
y_1 , z_0,z_1),V(x_0,x_1,x_2,x_3,y_0,y_1,y_2
z_0,z_1),$$
$$V(x_0,x_1,x_2,x_3,x_4,y_0,y_1,y_2
z_0,z_1),V(x_0,x_1,x_2,x_3,x_4,y_0,y_1,y_2,y_3
z_0,z_1),$$
$$V(x_0,x_1,x_2,x_3,x_4,x_5,y_0,y_1,y_2,y_3
z_0,z_1,z_2),V(x_0,x_1,x_2,x_3,x_4,x_5,x_6,y_0,y_1,y_2,y_3,y_4
z_0,z_1,z_2),$$
$$V(x_0,x_1,x_2,x_3,x_4,x_5,x_6,x_6,x_8,y_0,y_1,y_2,y_3,y_4
z_0,z_1,z_2,z_3)\},$$
and those centered above the $z$ axis 
$$\{V(x_0,y_0),V(x_0,y_0,y_1)\}.$$

\end{lem}
From lemma \ref{ec}, we deduce directly the set $EE.$ First recall  
that a monomial valuation, defined on the ring $\Cc[x,y,z]$ and  
associated with a vector $a=(a_1,a_2,a_3)\in \mathbb{N}^3,$ is defined by:
for $h=\sum_{i=(i_1,i_2,i_3)\in \mathbb{N}^3} c_ix^{i_1}y^{i_2}z^{i_3} \in 
\Cc[x,y,z]$ 
then
\begin{equation}\label{mon}
\nu_{a}(h)=\text{min}_{i\in \mathbb{N}^3;c_i\not=0}~~a_1i_1+a_2i_2+a_3i_3.
\end{equation}

From lemma \ref{ec} we deduce:
\begin{prop}
The valuations that belong to $EE(A_n)$ are monomial. They are associated with the vectors $v(\mathcal{C}_m)$ where $\mathcal{C}_m \in EC(A_n).$  These vectors are 
$$(l,m-l+1,1), l=1,\ldots,m~~\mbox{and}~~m=1,\ldots, n$$
$$(l,0,1),l=1,\ldots,n+1,$$
$$(0,l,1),l=1,\ldots,n+1.$$

The valuations that belong to $EE(D_{2n})$ are monomial. 
They are associated with the vectors $v(\mathcal{C}_m)$ where $\mathcal{C}_m \in 
EC(D_{2n}).$  These vectors are  \\
 $$U_1=(1,1,1), U_2=(1,2,2),\ldots ,U_{n-1}=(1,n-1,n-1),$$
 $$W_1=(1,1,2), W_2= (1,2,3),\ldots ,W_{n-1}=(1,n-1,n),$$
 $$V_1=(2,1,2),V_2=(2,2,3),\ldots V_{2n-2}=(2,2n-2,2n-1),$$
 $$W_0=(1,0,1),V_0=(2,0,1).$$
 
The valuations that belong to $EE(E_6)$ are monomial. They are associated with the vectors $v(\mathcal{C}_m)$ where $\mathcal{C}_m \in EC(E_6):$ \\

$$(1,1,1),(1,1,2),(1,2,2),(2,2,3),(2,3,4),(3,4,6)$$ 

The valuations that belong to $EE(E_7)$ are monomial. They are associated with the vectors $v(\mathcal{C}_m)$ where $\mathcal{C}_m \in EC(E_7):$ \\
 $$(1,1,0),(1,2,0),(1,1,1),(2,1,1),(2,2,1),(3,2,1),(3,3,1),(3,2,2),$$
$$(4,3,2),(5,3,2),(5,4,2),(6,4,3),(7,5,3),(9,6,4)$$ 
 
The valuations that belong to $EE(E_8)$ are monomial. They are associated with the vectors $v(\mathcal{C}_m)$ where $\mathcal{C}_m \in EC(E_8):$ \\
 $$(1,1,1),(1,1,2),(1,2,2),(1,2,3),(2,2,3),(2,3,4),(2,3,5),(2,4,5),(3,4,6),$$
 $$(3,5,7),(3,5,8),4,6,8),(4,6,9),(4,7,10),(5,7,11),(5,8,11)(5,8,12), (5,9,13)$$ $$(5,9,14),(6,10,14), (6,10,15)  $$

\end{prop}

\begin{proof}
Let $a=(a_1,a_2,a_3)\in \mathbb{N}^3.$ Assume that 
$$\mathcal{C}_m=V(x_0,\ldots,x_{a_1-1},y_0,\ldots,y_{a_2-1},z_0,\ldots,z_{a_3-1}
)\subset \Cc^3_{m}$$ is an irreducible component  of $X_m, $ which 
belongs to $EC(X).$ Then, a generic point of 
${\Psi_{m}^a}^{-1}(\mathcal{C}_m)\cap Cont^{m+1}(f)$ is of the shape 
$(x_{a_1}t^{a_1}+\cdots,y_{a_2}t^{a_2}+\cdots,z_{a_3}t^{a_3}+\cdots)$ where 
$(x_{a_1},y_{a_2},z_{a_3})$ is generic. For $h \in \Cc[x,y,z],$ it 
follows from  the definition of $\nu_{\mathcal{C}_m}(h)$ (see equation 
(\ref{val})) and from the definition of $\nu_{a}(h)$ (see equation (\ref{mon})) 
that $\nu_{\mathcal{C}_m}(h)=\nu_{a}(h).$ The proposition follows.
\end{proof}

\section{Toric minimal embedded resolutions}

In this section, we define   minimal toric embedded 
resolutions, and we give the motivations for this definition.\\ 

Let $N\backsimeq \Z^3$ be a lattice and let $\sigma \subset 
N_{\R^3}=N\bigotimes_\Z \R$ be the strictly convex cone generated by the 
vectors $v_1,... ,v_r$, i.e. $$\sigma=<v_1,\ldots,v_r> = \{\sum_{i=1}^{r} \lambda_i v_i, \ 
\lambda_i \in \R_+\}.$$

Let $f=\sum_{i=(i_1,i_2,i_3)\in \mathbb{N}^3} c_ix^{i_1}y^{i_2}z^{i_3}$
 be a polynomial in $\Cc[x,y,z]$ with $f(0)=0$.  We denote by $NP^+(f)$  its Newton 
polyhedron at the origin: this is the 
convex hull in $\R_+^{3}$ of the set 
$$\{(i_1,i_2,i_3)+\R_+^{3},c_i\not=0\},$$  where we have supposed that $f$ is given by 
its Taylor expansion at the origin and let $NP(f)$ be the union of the compact boundaries of $NP^+(f)$, called Newton boundary. 
This polyhedron and its boundary depend on the 
choice of coordinates. Let $\Gamma(f)$ be the dual fan of $NP(f);$ we will call covectors 
vectors which lie in $\Gamma(f).$  A covector defines a linear map on $\R^3.$ 
For a positive covector $L$,  { i.e. with positive coefficients}, we define the distance $d(L,f)$ as the minimal value of the linear map $L$ at $x\in NP^+(f);$ 
let $\Delta(L,f):=\{ x\in NP^+(f): L(x)=d(L,f)\}$ be the dual face to $L.$  We can then define $f_L(x)$ as the restriction of $f$ to 
the dual face of $L$, $\Delta(L,f);$ that is $$f_L(x)=\sum_{i=(i_1,i_2,i_3)\in \Delta(L,f)} c_ix^{i_1}y^{i_2}z^{i_3}.$$

We can find more details about Newton 
polyhedra, their dual fans and Newton non-degeneracy in \cite{O2}(see also 
\cite{AGS}).\\

\begin{deft} Let $f\in k[x,y,z]$ be such that $f(0)=0.$ We say that the hypersurface $\{f=0\}$ is non-degenerate 
with respect  to  a covector  $L$ if the variety $$F^*(L)=\{u\in 
{k^*}^3 / f_L(u)=0\}$$ is a reduced smooth hypersurface in the torus 
$k^{*3}.$ We say that $f$ is non degenerate at the origin $0$ with respect to its Newton 
Polyhedron $NP(f)$ if it is non degenerate for each covector.
\end{deft}

Recall that a fan $\Sigma$  is said to be regular if every cone $\sigma\in \Sigma$ is regular, i.e. the 
 vectors generating $\sigma$ are part of a basis of $\mathbb{Z}^3$ as a $\mathbb{Z}-$module, or equivalently (for $\sigma$  a cone maximal dimension), if 
 $\sigma=<v_1,v_2,v_3>$ then the absolute value $\mid det(v_1,v_2,v_3)\mid=1.$
A very useful result about toric embedded resolutions of this type of singularities goes as follows. 

\begin{thm}\label{NND}(\cite{Va},\cite{O1},\cite{AGS})
 We consider a pair $X\subset \Cc^3$ where $X=\{f=0\}$ is a Newton non-{degenerate} singularity, then the folllowing properties are equivalent:\\
 1) A subdivison $\Sigma$ of $\Gamma(f)$ is regular,\\
 2) The proper birational morphism $\mu_\Sigma:Z_\Sigma \longrightarrow \Cc^3=Spec~k[x,y,z]$ induced by $\Sigma$ is an embedded resolution of singularities 
 of the pair. Here $Z_\Sigma$ is the toric variety associated with the fan $\Sigma.$
 \end{thm}

  Note that the morphism $\mu_\Sigma$ is explicit in term of $\Sigma;$ 
 more precisely, a regular cone $\sigma=<v_1,v_2,v_3>\subset \Sigma$ of maximal dimension determines an affine chart of $Z_\Sigma$ which is isomorphic to 
 the affine space 
 $\Cc^3=Spec~k[q,r,s]$ and the rectriction of $\mu_\Sigma$ to this chart is given by
 $$ x=q^{v_{1,1}}r^{v_{2,1}}s^{v_{3,1}}$$
 $$y=q^{v_{1,2}}r^{v_{2,2}}s^{v_{3,2}}$$
 $$z=q^{v_{1,3}}r^{v_{2,3}}s^{v_{3,3}},$$
where $v_i= (v_{i,1},v_{i,2},v_{i,3}).$ Moreover, an edge of $\Sigma$ (a cone of dimension $1$) determines an orbit of $Z_\Sigma$ whose Zariski closure
defines a divisor on $Z_\Sigma$ which is an irreducible component of the exceptional divisor of $\mu_\Sigma$ (\cite{O}).

And the  strict transform of $X=\{f=0\}$ is defined as follows:
\begin{deft}\label{st} Let $\Sigma$ be a regular subdivision of $\Gamma(f).$ 
The strict transform of $\{f=0\}$ by $\mu_\Sigma$ is the Zariski closure of $(\mu_\Sigma)^{-1} ({\C^*}^3 \cap \{f=0\}).$ 
\end{deft}
  
We now recall the notion of ``essential divisor''\textcolor{red}{,} which is one of the motivations of this work. For a singular variety, we can find infinitely many resolutions of
singularities; the notion of essential divisor searches for intrinsic data in all resolutions of singularities of a given singular variety.

\begin{deft}(Essential divisors)\\
Let $X$ be a singular variety and $\pi: (\tilde{X},E)\rightarrow (X,sing(X))$ be a resolution of singularities of $(X,Sing(X).$ Let $E_i \in E$ be a divisor
which is an irreducible component of the exceptional locus $E$.
 We say that $E_i$ is an essential divisor if for any other  resolution of singularities 
 $\pi':(X',E')\rightarrow(X,Sing\ X),$ the center on $X'$ of the divisorial valuation determined by $E_i$ is an irreducible component of $E'$. 
A divisor $E_i$ is said  inessential if it is not essential.
\end{deft}

It should be noted that the notion of essential divisor given above is associated with abstract resolutions of singularity and not 
with embedded resolutions of singularities. For surface singularities, a divisor is essential if it appears on the minimal resolution of singularities. Note that a minimal resolution does not exist in general for a
singular variety of dimension larger than $2.$ 

For normal toric singularities, it is possible to determine essential divisors. For that we need the following definition
of a minimal system of generators of a cone.\\
Let $N \cong 
\textbf{Z}^n$ 
be a lattice and let $N_{\textbf{R}}$ be the real vector space $N\otimes_\textbf{Z}\textbf{R}.$

\begin{deft}(\cite{BGS}, see also \cite{AM}) Let $\sigma \subset N_{\textbf{R}} $ be a strongly convex rational polyhedral  cone.
 The minimal system of generators of $\sigma$ is the set:
$$ G_\sigma =\{x\in \sigma\cap N\backslash{0}~~|~~ \forall n_1,n_2 \in  \sigma\cap 
N, x=n_1+n_2 \Rightarrow n_1=0 \ or\ n_2=0 \}$$

An element in $G_{\sigma}$ is also called irreducible. It is 
primitive by definition.
\end{deft}
It follows from propostion $1.3$ in \cite{BGS},  that the elements of $G_\sigma$ appear on every regular fan which subdivides $\sigma$ as an extremal vector,
i.e. a primitive generator of a $1-$dimensional cone of the fan. Note that such a subdivision determines a (toric) resolution of singularities of the toric variety
defined by $\sigma,$ and dimension $1$ cones determine the irreducible components of the exceptional divisors. This gives a feeling for the following theorem which 
characterizes essential divisors on toric varieties.  

\begin{thm}\label{BG}(\cite{BGS}, \cite{IK})\\
Let $\sigma \subset N_{\textbf{R}} $ be a strongly convex rational polyhedral  cone. The minimal system of generators $G_\sigma$ is in bijection with the set of essential divisors of  
the toric variety $V_\sigma$. Therefore a primitive vector is in $G_\sigma$ if and only if it is an extremal vector of any regular subdivision
of the cone $\sigma$. 
\end{thm}

Actually, the fact that the divisors which correpond to elements of $G_\sigma$ are toric essential (i.e. essential 
for toric resolutions of singularities) follows from \cite{BGS}  and the fact that they are essential for all resolutions of singularities follows from \cite{IK}. 
It also follows from \cite{BGS} that for a three dimensional toric variety defined by a cone $\sigma$ there exists a toric 
resolution of singularities which is associated with a subdivision of $\sigma$ all of whose $1$-dimensional cones are rays associated to elements of $G_\sigma.$ Such a resolution is called a $G_\sigma$-resolution.
We think of such a resolution as a minimal toric resolution. Note that such a resolution does not exist in general for normal toric varieties of dimension 
larger than $3.$ Together with theorems \ref{NND} and \ref{BG}, this discussion motivates the following two definitions:

\begin{deft}(Toric embedded-essential divisors)\\
 Let $X\subset \Cc^3$ where $X=\{f=0\}$ is a Newton non-degenrate singularity. A toric embedded essential divisor $E$ is a divisor which appears on every toric resolution of singularities of $X$(such a resolution exists thanks to theorem \ref{NND}) as an irreducible component of the exceptional divisor.
This is equivalent to say that $E$ corresponds to an element of $G_\sigma$ for some  cone  $\sigma\in \Gamma(f)$  of maximal dimension.
\end{deft}

\begin{deft}(Minimal toric embedded resolution) Let $X=\{f=0\}\subset  \Cc^3$ be a Newton non-degenerate singularity at the origin. 
A toric embedded resolution of $X\subset \Cc^3$  which is  associated with a subdivision $\Sigma$ of $\Gamma(f)$ 
is minimal if the only $1-$dimensional cones which  appear in $\Sigma$ are   determined by elements of $G_\sigma$ for some  cone  $\sigma\in \Gamma(f)$  of maximal dimension and the $1-$dimensional cones of $\Gamma(f).$ This is equivalent to say that all the irreducible components of the exceptional divisor of a minimal toric embedded resolution are toric embedded essential divisors.
\end{deft}
Note that these definitions make sense also for non-isolated singularities as in \cite{ACT} for instance.

\section{ Toric minimal embedded resolutions for simple singularities}

In this section, we construct toric embedded resolutions from the data of the jet schemes of simple singularities that we gave 
in section 2. We also prove that for all singularities of  this type, except for the $E_8$ singularity,
the essential components of the jet schemes (introduced in section 2) correspond to toric embedded-essential divisors. We will treat each type of {singularity} separately: first we construct a toric resolution of singularities using  the weight vectors associated in section 2 with essential components, then
we construct a regular subdivision of the dual fan of the singularity embedded in $\Cc^3,$ and finally by using a theorem from \cite{BGS} (see also \cite{AM}) which characterizes toric essential divisor 
for toric varieties, we will prove for simple singularities (except for the singularity of type $E_8$) the minimality of our toric embedded resolution of singularities.

\begin{remark}
Instead of drawing the dual fan of each polyhedron in $\R^3$ (with coordinates $x,y,z)$  we use the trace of the dual fan on the hyperplane defined by the equation $x+y+z=1.$ Vectors are then represented by points. We call this trace the {\bf Newton  face} or simply, with some abuse of notation {\bf dual fan}.
\end{remark}

\subsection{The $A_n$ singularities}
\begin{thm}
The weight vectors of $A_{n}$ give a   toric minimal embedded resolution of the singularity. In particular the set of essential components $EC(A_{n})$ is in bijection with the divisors which appear on every toric minimal embedded resolution of these singularities. 
\end{thm}

The Newton polyhedron and its dual fan are the following:

\begin{figure}[h]
\begin{center}
\setlength{\unitlength}{0.40cm}

\includegraphics[scale=0.4]{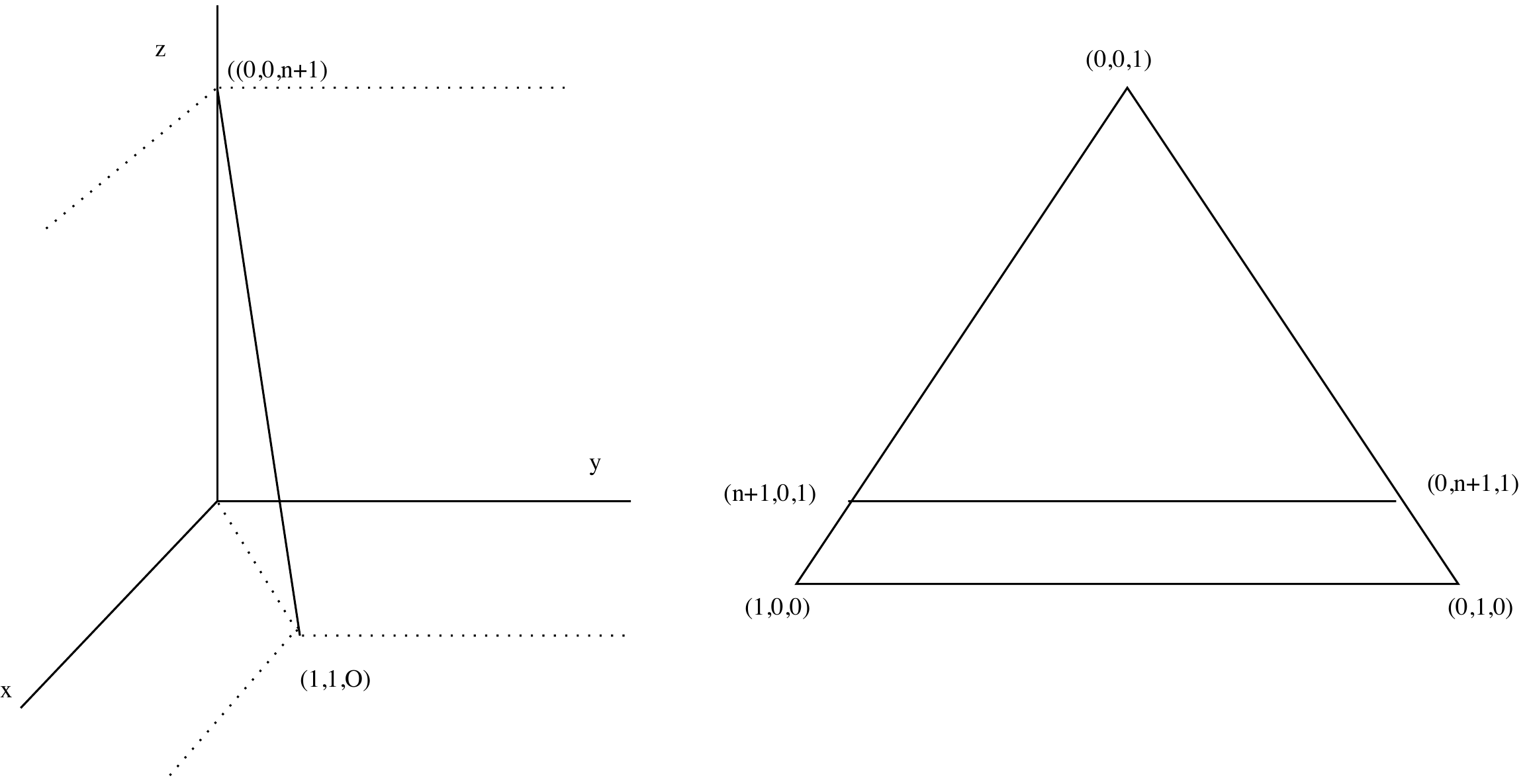}
\end{center}
\end{figure}

 The dual fan is not simplicial; a resolution of this singularity is given by  a regular and simplicial subdivision of the Newton dual fan.\\
{ The set  $EE(A_n) $ of valuations, found after a careful reading of {\cite{Mo1}}, is  $$(l,m-l+1,1), l=1,\ldots,m~~\mbox{and}~~m=1,\ldots, n$$
$$(l,0,1),l=1,\ldots,n+1,$$
$$(0,l,1),l=1,\ldots,n+1.$$     (see lemma $ 2.2$ and  proposition $2.3$). We reorder this set  in the following way in order to study them more easily (see the remarks below):}

${EE(A_n)}:=\{ (n,0,1),\dots (1,0,1),(0,n,1),\dots ,(0,1,1),(1,1,1),
(2,1,1),\dots ,(n,1,1),$
 $$(1,2,1)\dots (1,n,1) ,(2,2,1),(3,2,1)\dots (n-1,2,1),$$
$$(2,3,1)\dots (2,n-1,1)\dots (\lceil \frac{n+1}{2}\rceil,\lceil \frac{n+1}{2}\rceil,1)   \}      $$
to which we add , if $n$ is even,  $(\lceil \frac{n+1}{2}\rceil +1,\lceil \frac{n+1}{2}\rceil,1)$ et $(\lceil \frac{n+1}{2}\rceil ,\lceil \frac{n+1}{2}\rceil +1,1)$.\\


 The weight vectors are positioned on the following figure : \\\\

\begin{figure}[h]
\begin{center}
\setlength{\unitlength}{0.50cm}

\includegraphics[scale=0.4]{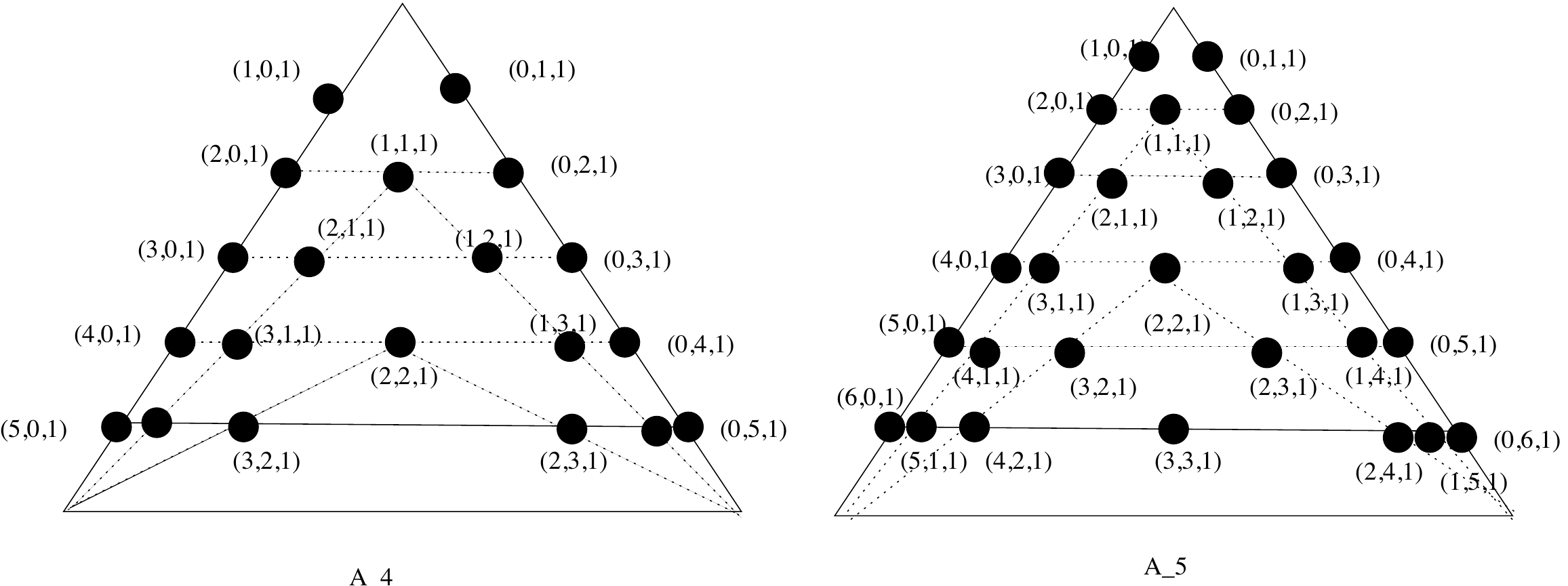}

\end{center}
\caption{ Positions of $ EE(A_4)$ and $ EE(A_5)$}
\end{figure}

{\bf CLAIM: these vectors are irreducible,  i.e. they represent essential components in the resolution.}\\
Before proving this claim, we study the position of the vectors in the Newton face of  dual fan, and show that they provide a toric embedded resolution.\\

{\bf Some remarks on the vectors of $EE(A_n)$}:\\
(with abuse of notation, the vectors are considered as points on the Newton face of the cone $[(e_3,(n+1,0,1),(0,n+1,1)]$).\\

\begin{enumerate}
\item We first prove that the vectors are rational sums of $e_1,e_2,e_3,v_1=(n+1,0,1),v_2=(0,n+1,1)$. By the symmetry of the singularity, we only have to look at half of the vectors:\\
we have : $$(k,l,1)=\frac{k}{n+1}v_1+\frac{l}{n+1}v_2+\frac{n-k-l+1}{n+1}e_3$$
for all $1\le k\le n$ and $1\le l \le n-k+1$.\\\\

Moreover $(k,0,1)=\frac{k}{n+1}v_1+\frac{n}{n+1}e_3$.\\
This implies that they all belong to the cone $(e_3,(n+1,0,1),(0,n+1,1))$.
\item One can remark the symmetry we obtain. It will follow that we have to do only half of the computations;
\item The vectors $(1,1,1)\dots (n,1,1)$ are on the line $((1,1,1),(1,0,0))$ as $(k,1,1)=(1,1,1)+(k-1)(1,0,0)$.  The same holds for $(k,k,1)\dots (n+1-k,k,1)$ for $k\le \lceil \frac{n+1}{2}\rceil$ ;
\item Homogeneous vectors $(\alpha,\beta,\gamma)$ such that $\alpha +\beta+ \gamma=k$, for $3\le k \le n+2$ are also on the same line.
\end{enumerate}

All these remarks show us a way to obtain an almost symmetric resolution from the vectors. Many algorithms  work, we propose {one of them} to the reader.\\\\

{\bf ALGORITHM}\\
First place $(1,1,1)$ and make the triangle $((0,0,1),(n+1,0,1),(0,n+1,1))$ simplicial.\\
Then place successively the vectors $(k,0,1),\dots (1,0,1)$ , joining each new vector to $(1,1,1)$, to make the resolution simplicial. Do the same symmetrically with $(0,k,1),\dots, (0,1,1)$.\\
One has $$(k,0,1)=\frac{k}{n+1}(n+1,0,1)+\frac{n+1-k}{n+1} (0,0,1)$$
Thus these vectors are primitive vectors.\\

For $A_{2n}$: \\
Place the vectors $(n,1,1)$ and $(1,n,1)$ and the corresponding edges. One can remark that the vectors $(k,1,1)$ lie on the line $(n,1,1),(1,1,1))$. It is at this point that the resolution is not symmetrical: if one chooses  $(n,1,1)$, then one makes the trapezium simplicial by adding the edges $((n,1,1),(0,1,0))$ and $((n,1,1),(1,0,0))$. So in the toric resolution, one will have the edges joining the vectors on $((n+1,0,1),(0,n+1,1))$ to $(0,1,0)$. Conversely, if one  chooses $(1,n,1)$, one will have the edges joining the vectors on $((n+1,0,1),(0,n+1,1))$ to $(1,0,0)$.\\
For $A_{2n+1}$, one had $[(\frac{n+1}{2},\frac{n+1}{2},1)] $ first, the trapezium thus becomes divided into three parts and the resolution become symmetrical.\\

Then add alternatively $(2,2,1), (n-1,2,1),(2,n-1,1),(3,3,1)\dots (\lceil \frac{n+1}{2}\rceil ,\lceil \frac{n+1}{2}\rceil,1)$,and if $n$ is even $(\lceil \frac{n+1}{2}\rceil+1 ,\lceil \frac{n+1}{2}\rceil ,1)$,$(\lceil \frac{n+1}{2}\rceil ,\lceil \frac{n+1}{2}\rceil +1,1)$ .
Then add the remaining vectors in the natural order, with the corresponding edges.\\\bigskip

\begin{figure}[h]
\begin{center}
\setlength{\unitlength}{0.50cm}

\includegraphics[scale=0.8]{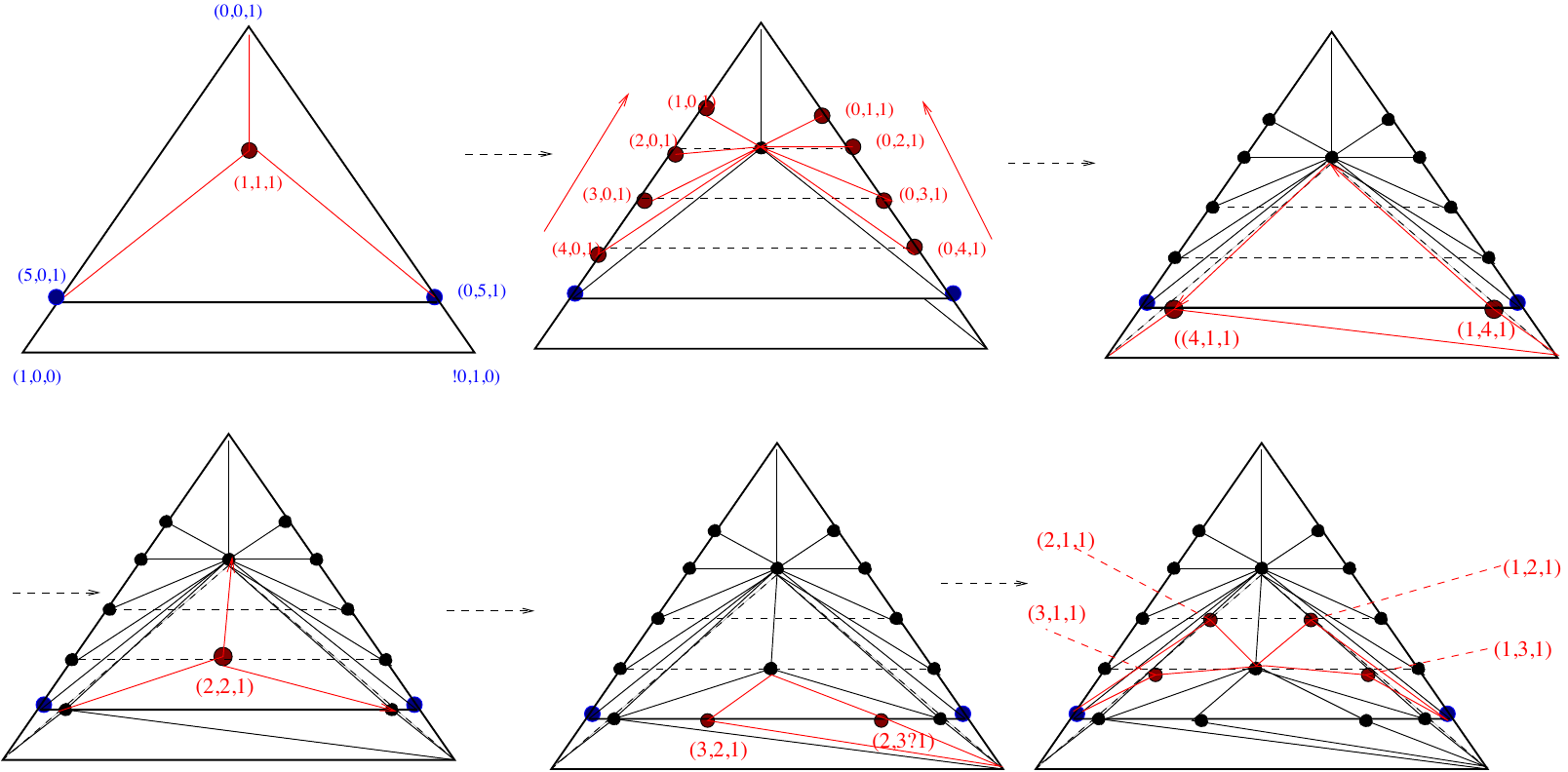}

\end{center}
\caption{ Resolution of $A_4$ step by step}
\end{figure}



\newpage

{\bf CLAIM : This is  a toric resolution for $A_{2n}$.}\\

\begin{proof}
\begin{itemize}
\item Let us look at triangles $[(1,1,1),(k,0,1),(k+1,0,1)]$ and $[(1,1,1),(0,k,1),(0,k+1,1)]$ for  $1\le k\le n,$  \\
we have:
$$\left \vert \begin{array}{ccc}
\  1&0&0\\
\ 1&k&k+1\\
\ 1&1&1\\
\end{array}\right \vert= -1\ \ \ 
and \ \ \ \left \vert \begin{array}{ccc}
\  1&k&k+1\\
\ 1&0&0\\
\ 1&1&1\\
\end{array}\right \vert= 1$$
\item We also have the triangles  $[(1,0,1),(1,1,1),(0,0,1)]$ et $(0,1,1),(1,1,1),(0,0,1)]$ of determinant $1$.
\item Now, we look at  $[(k,n-k,1),(k-1,n-(k-1),1),(0,1,0)]$.\\
We have
$$\left \vert \begin{array}{ccc}
\  k&k-1&0\\
\ n-k&n-(k-1)&1\\
\ 1&1&0\\
\end{array}\right \vert= 1$$

If $n$ is even, we have the triangles t $[(\lceil \frac{n+1}{2}\rceil -1, \lceil \frac{n+1}{2}\rceil,1),(\lceil \frac{n+1}{2}\rceil ,\lceil \frac{n+1}{2}\rceil +1,1),(0,1,0)]$ and if $n$ is odd, we have the triangles $[(\lceil \frac{n+1}{2}\rceil +1, \frac{n+1}{2},1),(\frac{n+1}{2},\frac{n+1}{2},1),(0,1,0)]$ et $[( \frac{n+1}{2}, \frac{n+1}{2}-1,1),(\frac{n+1}{2},\frac{n+1}{2},1),(0,1,0)]$ with determinant $1$.
\item The cones $[e_1,(n+1,0,1),(n,1,1)]$ and $[e_1,e_2,(n,1,1)]$ are non singular.
\item Finally we have to look at the triangles $[(k+l,k,1),(k+l-1,k,1),(k+1,k+1,1)]$ and $[(k+l,k,1),(k+l-1,k,1),(n+2-k,k-1,1)]$ for all $1\le k\le \lceil \frac{n+1}{2}\rceil$ and $k\le l \le n-1$(and to the symmetrical one too).\\
One has 

$$\left \vert \begin{array}{ccc}
\  k+l&k+l-1&k+1\\
\ k&k&1\\
\ 1&1&1\\
\end{array}\right \vert= 1    \ \ \ and\ \ \  \left \vert \begin{array}{ccc}
\  k+l&k+l-1&n+2-k\\
\ k&k&k-1\\
\ 1&1&1\\
\end{array}\right \vert= 1$$

\end{itemize}
\end{proof}

Now we have to prove that in fact this resolution is minimal or a $G-resolution$ (according to C. Bouvier and G. Gonzalez-Sprinberg).\\
First of all, the weight vectors $(l,m-l+1,1), l=1,\ldots,m~~\mbox{and}~~m=1,\ldots, n$ are clearly primitive as their third component is $1.$\\
Consider the simplicial cone generated by $e_3,(n+1,0,1),(0,n+1,1)$;  all the weight vectors belong to this cone.\\
Let us give the same names to the points in $\R^3$ which are the end of the vectors starting from $0$ and the vectors themselves.
\begin{lem}
The points in $\R^3$ $\{e_3,(n+1,0,1),(0,n+1,1),(l,m-l+1,1), l=1,\ldots,m~~\mbox{and}~~m=1,\ldots, n\}$ belong to the face $(e_3,(n+1,0,1),(0,n+1,1))$ of the thetahedron $[0,e_3,(n+1,0,1),(0,n+1,1)]$. This implies that
 the set $\{e_3,(n+1,0,1),(0,n+1,1),(l,m-l+1,1), l=1,\ldots,m~~\mbox{and}~~m=1,\ldots, n\}$ is free over $\Z$, i.e. none of the vectors is reducible into other vectors of this set.\end{lem}

Suppose {that} the lemma is {proved.} {Then} the only primitive vectors that could generate the characteristic vectors should be vectors inside the {tetrahedron} $[0,e_3,(n+1,0,1),(0,n+1,1)]$.  Suppose that there is such a vector. {By}  th. 1.10 of \cite{BGS}, it corresponds to an essential divisor in the toric resolution of the cone $[e_3,(n+1,0,1),(0,n+1,1)]$, { i.e.,} it appears in every regular subdivision as an extremal vector. But  the vectors  
 $(l,m-l+1,1), l=1,\ldots,m~~\mbox{and}~~m=1,\ldots, n\}$, we get such a regular subdivision, so there is no primitive vector in the tetrahedron $[0,e_3,(n+1,0,1),(0,n+1,1)]$.\\\\
Let us prove the lemma:
 \begin{proof} 
 The equation of the face $(e_3,(n+1,0,1),(0,n+1,1))$ is $z-1=0$. Clearly all the points $\{e_3,(n+1,0,1),(0,n+1,1),(k,n-k,1),(k-1,n-(k-1),1),(0,1,0),(l,m-l+1,1), l=1,\ldots,m~~\mbox{and}~~m=1,\ldots, n\}$ belongs to it.
  \end{proof}

\subsection{The $D_n$ singularities}
\begin{thm}
{The w}eight vectors of $D_{n}$ give a "\it{canonical}"  toric minimal embedded resolution of the singularity. In particular the set of essential components $EC(D_{n})$ is in bijection with the divisors which appear on every toric minimal embedded resolution of these singularities. 
\end{thm}

\begin{proof}
Let us do it for $D_{2n}$.\\
Let us recall what are the Newton polyhedron for $D_{2n}$ and the associated  dual cone:

\begin{figure}[h]
\begin{center}
\setlength{\unitlength}{0.50cm}

\includegraphics[scale=0.4]{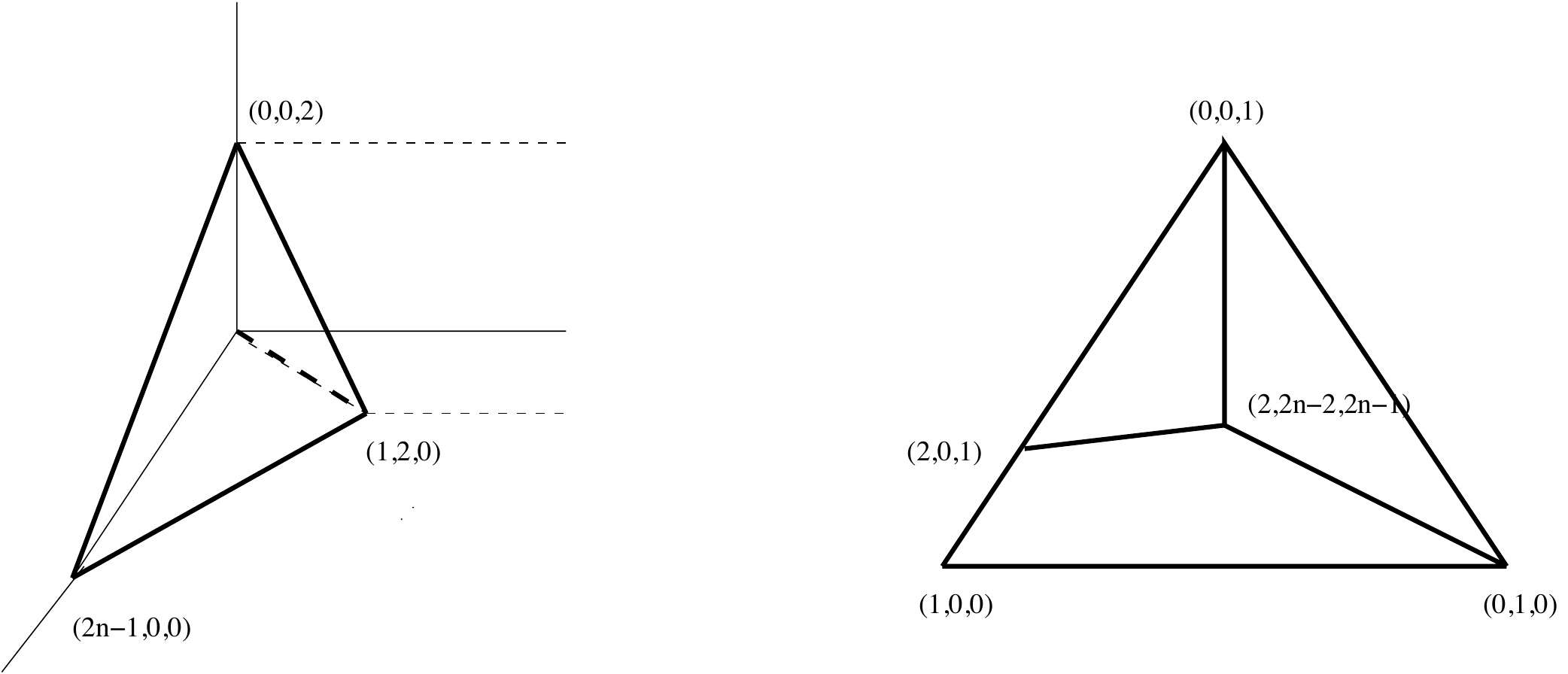}
\end{center}
\caption{ Newton polyhedron of $D_{2n}$ and its dual cone}
\end{figure}

First we make the dual polyhedron simplicial by adding the edge $[e_1,(2,2n-2,2n-1]$. Let us call $\mathcal{C}_1$ be the cone $[e_1,(2,0,1), (2,2n-2,2n-1)]$ , $\mathcal{C}_2$ be   the cone $[e_1,e_2, (2,2n-2,2n-1)],$ $\mathcal{C}_3$ be the cone $[e_3,(2,0,1), (2,2n-2,2n-1)]$ , $\mathcal{C}_4$ be the cone $[e_2,e_3, (2,2n-2,2n-1)]$ .\\

First of all, to obtain a toric resolution of singularities of $D_{2n}$ with equation $z^2+x(y^2+x^{2n-2})=0$, as it is not commode, one has to blow up the $y$ axes, i.e. to add the vector $(1,0,1)$.\\
Let us now consider the weight vectors and show some of their properties:
\begin{itemize}
\item The vectors $U_i=(1,i,i)$ ($1\le i \le n-1$) are {primitive} and belong to the cone $\mathcal{C}_2$. More precisely they lie onto the line $(e_1,(0,1,1))$:\\
In fact we have the following equalities:
$$U_i=\frac{2n-2i-1}{2n-1}e_1+\frac{i}{2n-1}e_2+\frac{i}{2n-1}(2,2n-2,2n-1)$$
and if we note $U_0=(1,0,0)$  , then for $1\le i \le n-1$, $$U_i=U_{i-1}+(0,1,1)$$
\item  The vectors $W_i=(1,i,i+1)$ ($1\le i \le n-1$) are {primitive} and belong to the cone $\mathcal{C}_3$. More precisely they lie onto the line $((1,0,1),(0,1,1))$:\\
In fact we have the following equalities:
$$W_i=\frac{n-1}{2n-2}e_3+\frac{n-i-1}{2n-2}(2,0,1)+\frac{i}{2n-2}(2,2n-2,2n-1)$$
 and  if we note $W_0=(1,0,1)$  , then for $1\le i \le n-1$, $W_i=W_{i-1}+(0,1,1)$
\item  The vectors $V_i=(2,i,i+1)$ ($1\le i \le 2n-2$) are primitive and belong to the cones $\mathcal{C}_1$ and $\mathcal{C}_3$. More precisely they lie onto the line $((2,0,1),(0,1,1))$:\\
In fact we have the following equalities:
$$V_i=\frac{2n-2-i}{2n-2}(2,0,1)+\frac{i}{2n-2}(2,2n-2,2n-1)$$
 and, if we note $V_0=(2,0,1)$  , then for $1\le i < 2n-2$,  $V_i=V_{i-1}+(0,1,1).$
\end{itemize}

{\bf CLAIM: the simplicial decomposition of the dual Newton cone of $D_{2n}$ obtained by adding, the edge $[e_1,(2,2n-2,2n-1]$, the vector $(1,0,1)$, then the vectors $U_1$ to $U_{n-1}$ in the natural order, { the} vectors $W_1$ to $W_{n-1}$ in the natural order and the vectors $V_1$ to $V_{2n-2}$ in the natural order too, give a toric minimal embedded resolution of $D_{2n}$}.\\
The resolved dual cone is of the form:
\begin{figure}[h]
\setlength{\unitlength}{0.40cm}
\begin{center}
\includegraphics[scale=0.55]{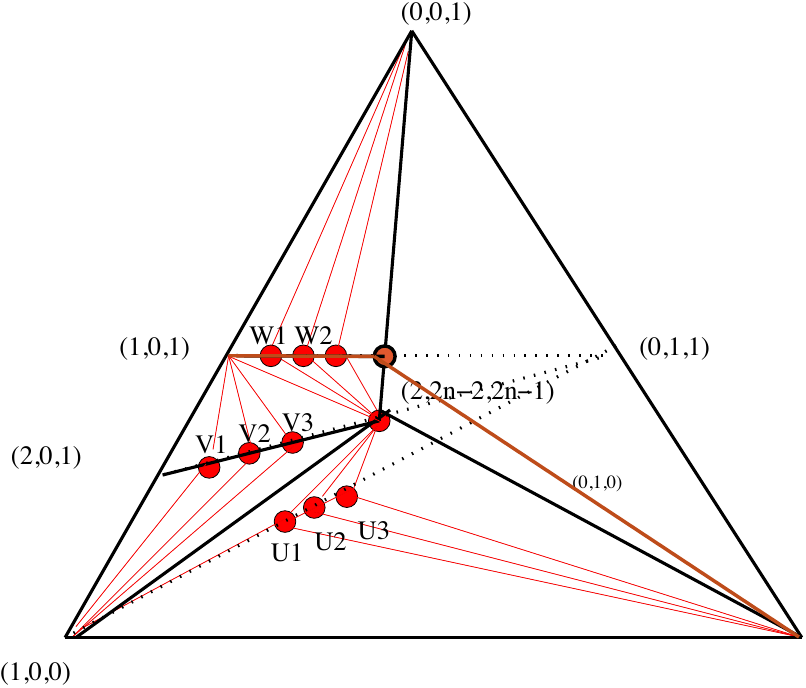}
\caption{Resolution of $D_{2n}$}
\end{center}
\end{figure}

Proof of the claim: by the previous observations, the  {vectors $U_i,\ V_i,\ W_i$ above} are primitive; thus it remains to prove that the cones obtained are non singular, i.e. they have determinant $+1$ or $-1$ and using the same proof as for $A_n$, we prove then that the vectors are irreducible. \\\\
We add first the vector $(1,0,1)$, which corresponds to the blowing up of the $y$-axes. Then, the cone $\mathcal{C}_3$ is decomposed into two cones :  $\mathcal{C}_{31}=((1,0,1),(2,0,1),(2,2n-2,2n-1))$ and  $\mathcal{C}_{32}=((1,0,1),e_3,(2,2n-2,2n-1))$ . By the above observation, we have now that each group of characteristic vectors are in different cones. 

  \begin{figure}[h]
\setlength{\unitlength}{0.40cm}
\begin{center}
\includegraphics[scale=0.4]{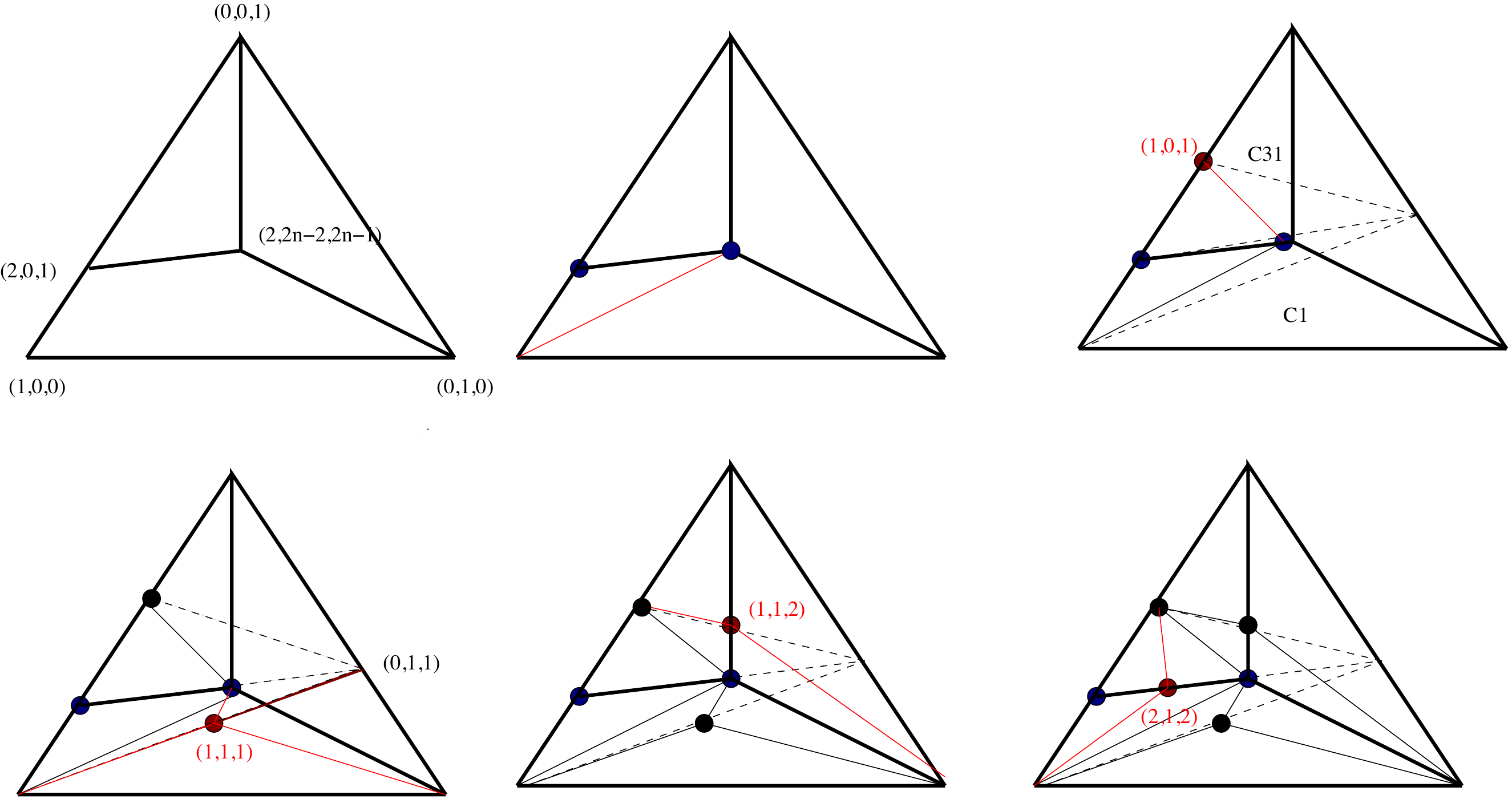}
\caption{Resolution of $D_{4}$ step by step}
\end{center}
\end{figure}

\begin{itemize}
\item Decomposition of the cone $\mathcal{C}_1$:\\
Let us note $U_0=e_1$.\\
We add the vectors from $U_1$ to $U_{n-1}$ and show by induction that we obtain at the end regular cones. At each step the cone where lie the new $A_i$ is decomposed into three cones: $(U_i, U_{i-1},e_2)$, $((U_i, U_{i-1},(2,2n-2,2n-1))$ and $(U_i,e_2,(2,2n-2,2n-1)$. That last one contains $U_{i+1}$ .
$$\left \vert \begin{array}{ccc}
\  1&1&0\\
\ i&i-1&1\\
\ i&i-1&0\\
\end{array}\right \vert =1,\  \left \vert \begin{array}{ccc}
\  1&1&2\\
\ i&i-1&2n-2\\
\ i&i-1&2n-1\\
\end{array}\right \vert = -1,\left \vert \begin{array}{ccc}
\  1&2&0\\
\ i&2n-2&1\\
\ i&2n-1&0\\
\end{array}\right \vert = n-1-2i$$

That last determinant is not equal to  $+1$ or $-1$, so we add $U_{i+1}$ and so on.\\
 For $i=n-1$ the last determinant is equal to one.

\item Decomposition of the cone $\mathcal{C}_{31}$:\\
Let us note $W_0=(1,0,1)$.\\
We add the vectors from $W_1$ to $W_{n-1}$ and show by induction that we obtain at the end regular cones. At each step the cone where lie the new $W_i$ is decomposed into three cones: $(W_i, W_{i-1},e_3)$, $((W_i, W_{i-1},(2,2n-2,2n-1))$ and $(W_i,e_3,(2,2n-2,2n-1)$. That last one contains $W_{i+k}$  for $k\ge 1$.\\
We have :
$$\left \vert \begin{array}{ccc}
\  1&1&0\\
\ i&i-1&0\\
\ i+1&i&1\\
\end{array}\right \vert = -1,\ \left \vert \begin{array}{ccc}
\  1&1&2\\
\ i&i-1&2n-2\\
\ i+1&i&2n-1\\
\end{array}\right \vert =1,\ \left \vert \begin{array}{ccc}
\  1&2&0\\
\ i&2n-2&0\\
\ i+1&2n-1&1\\
\end{array}\right \vert= 2n-2-2i$$
That last determinant is not equal to  $+1$ or $-1,$ so we add $A_{i+1}$ and so on.\\
 For $i=n-1$ the last determinant is equal to one.

\item Decomposition of the cones $\mathcal{C}_{32}$ and $\mathcal{C}_2$:\\
Let us note $V_0=(2,0,1)$.\\
We add the vectors from $V_1$ to $V_{2n-2}$ and show by induction that we obtain at the end regular cones. At each step the two cones where lie the new $C_i$ are decomposed into two cones themselves: $(V_i, V_{i-1},e_1)$, $(V_i, V_{i-1},(1,0,1))$ and $(V_i,e_3,(2,2n-2,2n-1)$, $(V_i,(1,0,1),(2,2n-2,2n-1)$. Those last ones contain $V_{i+1}$  and are not regular.
$$\left \vert \begin{array}{ccc}
\  2&2&1\\
\ i&i-1&0\\
\ i+1&i&0\\
\end{array}\right \vert = 1$$

$$\left \vert \begin{array}{ccc}
\  2&2&1\\
\ i&i-1&0\\
\ i+1&i&1\\
\end{array}\right \vert= \left \vert \begin{array}{ccc}
\  2&2&1\\
\ i&2n-2&0\\
\ i+1&2n-1&0\\
\end{array}\right \vert= 2ni-i-2ni+2i-2n+2=i-2n+2$$

and

$$\left \vert \begin{array}{ccc}
\  2&2&1\\
\ i&2n-2&0\\
\ i+1&2n-1&1\\
\end{array}\right \vert= 2n-2-i$$

For $i=2n-3$ the two last determinants are equal to  $+1$ or $-1.$
\end{itemize}

\bigskip

Moreover we have obtained a $G$-resolution:\\
As for $A_n$ singularity,  one can show that $W_i$ (resp.$U_i$ and $V_i$) belongs to the base of the cone $(e_3,(2,0,1),(2,2n-2,2n-1))$ (resp. $(e_1,(2,0,1)(2,2n-2,2n-1))\cap(e_3,(2,0,1),(2,2n-2,2n-1))$, and $(e_1,e_2,(2,2n-2,2n-1))$ , whose equation is $y-z+1=0$ (resp. $y-z+1=0; x+y-z-1=0$ and $x+y-z-1=0$). With these vectors we obtain a regular subdivision of each cone, which implies that there are no other primitive vectors inside each cone and that the characteristic vectors are irreducible. So we also obtain a $G$-resolution for $D_{2n}$.
\end{proof}

Remark: the fact that each vector "belongs" to the base of the cone is equivalent to {saying} that the volume of the subdivisions stay the same during the process of the resolution.


\subsection{The $E_6$ singularity, {according to} \cite{Mo4}} 

Now consider the singularity $E_6$ given by the equation $z^2+y^3+x^4=0$. A  simplicial dual Newton polyhedron associated to it is given on figure $6$.

\begin{thm}
The set of  weight vectors $EE(E_6)$ is $$(1,1,1),(1,1,2),(1,2,2),(2,2,3),(2,3,4),(3,4,6)$$ 
Call them in the lexicographic order.\\
{The w}eight vectors of $E_6$ give a "\it{canonical}"  toric minimal embedded resolution of the singularity. In particular the set of essential components $EC(E_6)$ is in bijection with the divisors which appear on every toric minimal embedded resolution of these singularities. 

\end{thm}
\begin{figure}[h]
\setlength{\unitlength}{0.4cm}
\begin{center}
\includegraphics[scale=0.45]{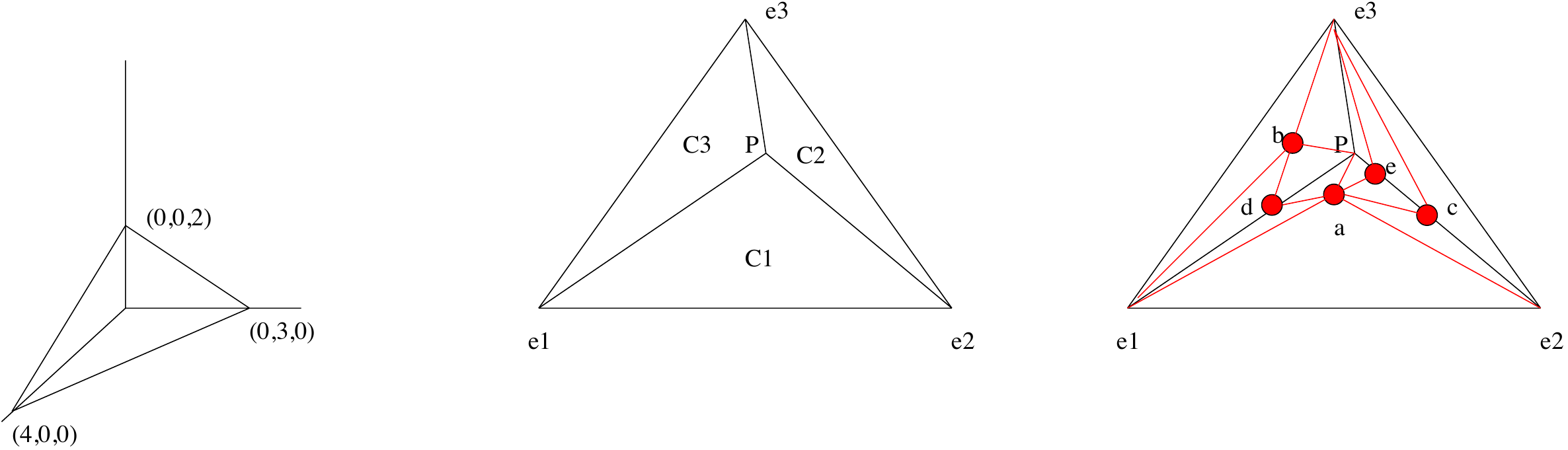}
\end{center}
\caption{ Newton dual polyhedron of $E_6$ and its resolution}
\end{figure}

\begin{proof}
The vectors are in different {subcones} of the dual fan. To make the fan regular, we divide each cone {using} the vectors. So {we have to look first} at the positions of each vector. One has :
$$\{(1,1,1),(1,2,2),(2,2,3),(2,3,4)\} \subset C_ 1\}$$
$$\{(2,3,4),(1,2,2)\} \subset C_ 2\}$$
$$\{(1,1,2),(2,2,3)\} \subset C_ 3\}$$

As for the previous singularities, one can prove that adding the weight vectors  in each cone (in the alphabetic order, see fig. 6), gives a minimal resolution for each of them. See fig. 6 for the picture.

\end{proof}

\subsection{The $E_7$ singularity}

Now consider the singularity $E_7$ given by the equation $x^2+y^3+yz^3=0$. A  simplicial dual Newton polyhedron associated to it is given on figure $7$.

\begin{thm}
The set of weight vectors is $$(1,1,0),(1,2,0)=O_2,(1,1,1),(2,1,1),(2,2,1),(3,2,1),(3,3,1),$$ $$(3,2,2),(4,3,2),(5,3,2),(5,4,2),(6,4,3),(7,5,3),(9,6,4)=O_1$$ Call them in the lexicographic order (except $O_1$ and $O_2$).

{The w}eight vectors of $E_7$ give a "\it{canonical}"  toric minimal embedded resolution of the singularity. In particular the set of essential components $EC(E_7)$ is in bijection with the divisors which appear on every toric minimal embedded resolution of these singularities.

\end{thm}
\begin{figure}[h]
\begin{center}
\setlength{\unitlength}{0.50cm}

\includegraphics[scale=0.6]{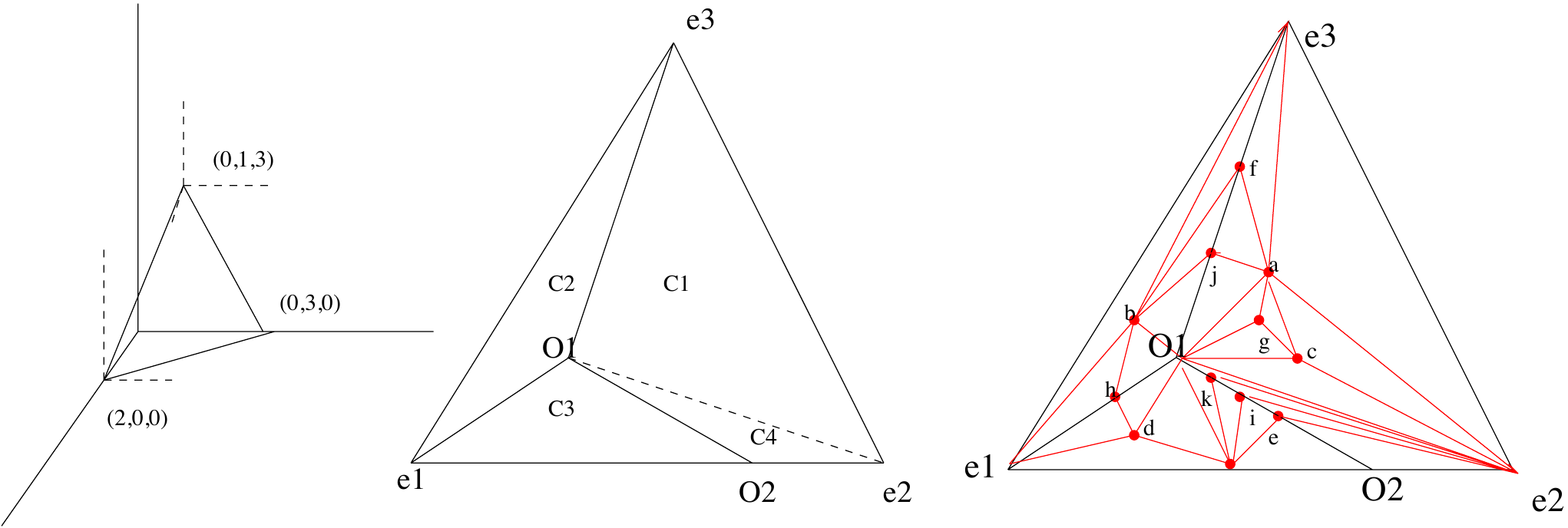}
\end{center}
\caption{ Newton dual polyhedron of $E_7$ and its resolution}
\end{figure}

\begin{proof}
The vectors are in different sub cones of the dual fan (see fig.7 for denomination of the cones). To make the fan regular, we divide each cone thanks to the vectors. So first we have to look at the positions of each vector. One has :
$$\{(1,1,1),(2,2,1),(3,2,2),(4,3,2),(6,4,3)\} \subset C_ 1\}$$
$$\{(2,1,1),(3,2,2),(5,3,2),(6,4,3)\} \subset C_ 2\}$$
$$\{(1,1,0),(3,2,1),(3,3,1),(5,3,2),(5,4,2),(7,5,3)\} \subset C_ 3\}$$
$$\{(3,3,1),(5,4,2),(7,5,3)\} \subset C_4 \}$$

As for the previous singularities, one can prove that adding the previous vectors  in each cone (in the lexicographic order), gives a minimal resolution for each of them. See fig.7 for the picture.

\end{proof}

\subsection{The $E_8$ singularity}

Now consider the singularity $E_8$ given by the equation $z^2+y^3+x^5=0$. A  simplicial dual Newton polyhedron associated to it is given on figure $8$. 

\begin{prop}
The weight vector gives an embedded resolution of $E_8$ but this resolution is not minimal. 
\end{prop}
\begin{proof}
Recall that the set of weight vectors are in this case : $$(1,1,1),(1,1,2),(1,2,2),(1,2,3),(2,2,3),(2,3,4),(2,3,5),(2,4,5),(3,4,6),$$
 $$(3,5,7),(3,5,8),4,6,8),(4,6,9),(4,7,10),(5,7,11),(5,8,11)(5,8,12), (5,9,13)$$ $$(5,9,14),(6,10,14), (6,10,15)=O  $$

\begin{figure}[h]
\setlength{\unitlength}{0.50cm}
\begin{center}
\includegraphics[scale=0.6]{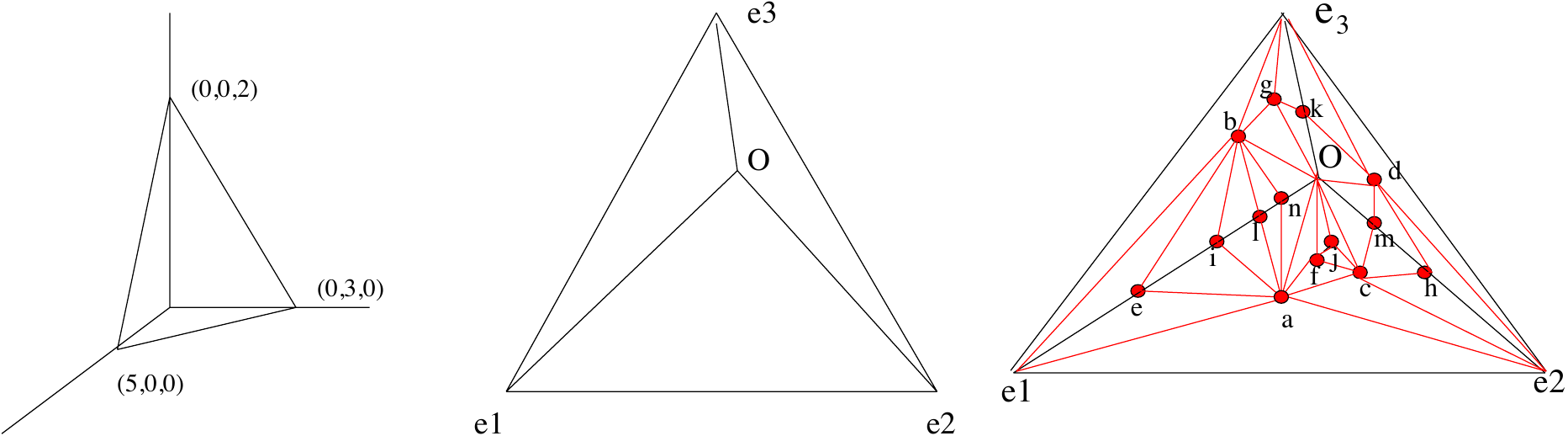}
\end{center}
\caption{ Newton dual polyhedron of $E_8$ and its resolution}
\end{figure}
One can show that we get a regular subdivision of each cone thanks to the following characteristic vectors $$(1,1,1),(1,1,2),(1,2,2),(1,2,3),(2,2,3),$$ $$(2,3,4),(2,3,5),(2,4,5),(3,4,6),(3,5,7),(3,5,8),(4,6,9),(4,7,10),(5,8,12), (6,10,15)=O$$ (We call them in the lexicographical order on fig.$8$)\\
 Moreover they are irreducible (one can show it {in the same way} as above), so they give subdivision which induces a minimal embedded toric resolution for $E_8$.\\
 In fact we have \begin {itemize}
 \item $(4,6,8)=(2,3,4)+(2,3,4)$
 \item $(5,7,11)=(4,6,9)+(1,1,2)$
 \item $(5,8,11)=(4,7,10)+(1,1,1)$
 \item $(5,9,13)=(3,5,8)+(2,4,5)$
 \item $(6,9,13)=(3,4,6)+(3,5,7)$
 \item $(5,9,14)=(3,4,6)+(3,5,8)$
 \item $(6,10,14)=(3,5,7)+(3,5,7)$
 \end{itemize}

\end{proof}

\subsection{A remark on abstract resolution of singularities of simple singularities}

For a simple singularity, the restriction of a minimal embedded resolution to the strict transform
of the singularitiy gives its minimal abstract resolution; note that not 
all the divisors that appear on the minimal embedded resolutions meet the strict transform, only those which
corresponds to weight vectors which belong to the two dimensional cones of the Newton dual fan. Moreover, these last
divisors may meet the strict tansform along more than one irreducible component.



Hussein Mourtada,\\
Equipe G\'eom\'etrie et Dynamique, \\
Institut Math\'ematique de Jussieu-Paris Rive Gauche,\\
Universit\'e Paris 7, \\
B\^atiment Sophie Germain, case 7012,\\
75205 Paris Cedex 13, France.\\
Email: hussein.mourtada@imj-prg.fr\\

Camille Pl\'enat\\
I2M, Groupe AGT, Aix Marseille Universit\'e \\ 
CMI, Technop\^ole Ch\^ateau-Gombert\\
39, rue F. Joliot Curie, 13453 Marseille Cedex 13\\
Email: camille.plenat at univ-amu.fr 

\end{document}